\documentclass[11pt, twoside, reqno]{amsart}
\usepackage[a4paper,left=35mm,right=35mm,top=30mm,bottom=30mm,marginpar=25mm]{geometry} 

\usepackage{xcolor}    
\usepackage{hyperref}
\definecolor{darkcyan}{cmyk}{1, 0, 0, 0.6}
\hypersetup{
    colorlinks=true,
    linkcolor=darkcyan,  
    urlcolor=darkcyan,    
    citecolor=darkcyan 
}
\makeatletter

\usepackage{subfig}
\usepackage{glossaries}
\usepackage{glossary-mcols}

\usepackage{aurical}
\usepackage{amsfonts,amsmath}	
\usepackage{amssymb}
\usepackage{verbatim}
\usepackage{amsopn}
\usepackage[english]{babel}
\usepackage{amsthm}
\usepackage{enumerate}
\usepackage{mathrsfs}	
\usepackage{enumitem}
\usepackage{mathtools}
\usepackage{esint}
\usepackage{bbm}
\usepackage{caption}
\usepackage{marginnote}
\usepackage{comment}

\theoremstyle{definition} \newtheorem{definition}{Definition}[section]
\theoremstyle{definition} \newtheorem{remark}[definition]{Remark}
\theoremstyle{plain} 
\theoremstyle{plain} \newtheorem{proposition}[definition]{Proposition}
\theoremstyle{plain} \newtheorem{theorem}[definition]{Theorem}
\theoremstyle{plain} \newtheorem{corollary}[definition]{Corollary}
\theoremstyle{definition} 
\theoremstyle{plain} 
\theoremstyle{definition}

\DeclareMathOperator{\dist}{dist}

\newcommand{\Id}{\mathrm{id}}

\renewcommand{\b}{\mathbf b}

\newcommand{\rest}{\llcorner}

\numberwithin{equation}{section} 

\theoremstyle{plain} \newtheorem*{theorem*}{Theorem}
\theoremstyle{plain} 
\theoremstyle{plain} \newtheorem*{mthm*}{Main Theorem}
\theoremstyle{plain} \newtheorem*{conjecture*}{Conjecture}
\theoremstyle{plain} \newtheorem{conjecture}[definition]{Conjecture}
\theoremstyle{plain} \newtheorem*{problem*}{Problem}
\usepackage{orcidlink}

\title{Measure preserving maps with bounded total variation}

\author{Stefano Bianchini}
\address{S. Bianchini: S.I.S.S.A., via Bonomea 265, 34136 Trieste, Italy}
\email{bianchin@sissa.it}

\author{Luca Talamini \orcidlink{0009-0008-5147-5184}}
\address{L. Talamini, S.I.S.S.A., via Bonomea 265, 34136 Trieste, Italy}
\email{ltalamin@sissa.it}

\begin{document}

\maketitle

\begin{abstract}
   Consider a piecewise affine Lipschitz map $\phi : \Omega \to \mathbb R$, where $\Omega \subset \mathbb R^d$ is an open set, and assume that $x \mapsto x + t \nabla \phi(x)$ is injective for almost every $t > 0$.  In (J.-G. Liu, R.~L. Pego, \emph{Rigidly breaking potential flows and a countable Alexandrov theorem for polytopes}, Pure Appl. Anal., \textbf{7}(4),  2025) the authors conjecture that every such $\phi$ must be locally convex. We prove the result assuming additionally $\nabla \phi \in BV_{loc}(\Omega)$, for a more general class of measure preserving maps. 
\end{abstract}

\section{Introduction}
In \cite{LP25} the authors consider the following problem: let $\Omega \subset \mathbb R^d$ be an open set, and consider masses $\{m_i\}_{i \in \mathbb N}$ and velocities $\{v_i\}_{i \in \mathbb N}$ such that
$$
\sum_i m_i = \mathscr L^d(\Omega)
$$
where $\mathscr L^d$ is the $d$-dimensional Lebesgue measure. It is a classical result in optimal transport \cite{Bre91} that there is a unique convex map $\nabla \phi$ such that 
$$
\mathscr L^d \Big(\big\{ x \in \Omega \; | \; \nabla \phi(x) = v_i \big\}\Big) = m_i.
$$
Since $\phi$ is convex, the map $x \mapsto x + t \nabla \phi(x)$ is injective for all $t > 0$ on the set $F \subset \Omega$, where
$$
F = \Big\{ x \in \Omega \; | \; \text{$\phi$ is differentiable at $x$} \Big\}.
$$
Clearly $F$ is of full measure in $\Omega$, i.e. $|\Omega \setminus F| = 0$.

\vspace{0.3cm}
Conversely, let a partition of $\Omega$ be given, up to a negligible set,
$$
\Omega = \bigcup_i A_i, \qquad \text{$A_i$ open sets}
$$
together with a Lipschitz map $\phi$, affine on each $A_i$, such that $x \mapsto x + t \nabla \phi(x)$ is injective on $F$ for a.e. $t > 0$. The following conjecture was put forward in \cite{LP25}:

\begin{conjecture}  
Every such  map $\phi$ must necessarily be locally convex in $\Omega$.
\end{conjecture}

\vspace{0.3cm}
Our aim here is to prove the result for maps $\phi$ with $\nabla \phi \in BV_{loc}$:
\begin{theorem}\label{thm:mainthm}
    In the above setting, if $\nabla \phi \in BV_{loc}(\Omega)$, $\phi$ must be locally convex.
\end{theorem}

\vspace{0.3cm}
Proving the result for $\nabla \phi \notin BV$ is presently an almost completely open problem. The only non-$BV$ case is proved in \cite{LP25} assuming only continuity of the gradient, i.e. $\nabla \phi \in C(\Omega; \mathbb R^d)$.

\vspace{0.3cm}
The paper is structured as follows. In Section \ref{sec:sub} we show that every such $\phi$ must be subharmonic, this holds without the $BV$ assumption.  In Section \ref{sec:main} we conclude the proof by showing, using the theory of $BV$ functions, that $\phi$ must be locally convex. In Section \ref{sec:optproof} we give a different proof of the subharmonicity of $\phi$, based on optimal transport tools, which yields an additional property of $\phi$ in the case $\nabla \phi \notin BV$.

\section{Measure preserving maps and main result}\label{sec:main}
Our Theorem \ref{thm:mainthm} will be a consequence of a more general result, valid for a class of measure preserving maps.

In all of the following,  $\Omega \subset \mathbb R^d$ is an open set.

\begin{definition}
Let $T$ be a measurable map $T : \Omega \to \mathbb R^d$ and $F \subset \Omega$ measurable be given. We say that $T$ is \emph{measure preserving} on $F$ if for all $A\subset \Omega$ measurable, $T(F \cap A)$ is measurable and
\begin{equation}\label{eq:mp}
    T_\sharp \mathscr L^d \llcorner (F\cap A) = \mathscr L^d \llcorner T(F\cap A).
\end{equation}
\end{definition}

\vspace{0.3cm}
\begin{definition}\label{defi:mpc}
    Let $\phi : \Omega \to \mathbb R$ be a $\mathbf W^{1,1}(\Omega)$ function. We say that $\phi$ satisfies the \emph{measure preserving condition} (in short (\textbf{MPC})) if there is a full measure set  $F \subset \Omega$ for which
\begin{equation}\label{def:Ttmp}
    T_t : \Omega \to \mathbb R^d, \qquad T_t(x) := x + t \nabla \phi(x)
\end{equation}
is measure preserving on $F$ for a.e. $t \geq 0$.
\end{definition}

\vspace{0.3cm}
We prove the following theorem. 
\begin{theorem}\label{thm:main}
    Let $\phi \in \mathbf W^{1,1}(\Omega)$ be a map satisfying \emph{(\textbf{MPC})}, and assume that $\nabla \phi \in BV_{loc}(\Omega, \mathbb R^d)$. Then $\phi$ is locally convex in $\Omega$.
\end{theorem}
Clearly, Theorem \ref{thm:mainthm} follows from \eqref{thm:main}. We would expect every function $\phi \in \mathbf W^{1,1}(\Omega)$ belonging to the more general class of maps satisfying the measure preserving condition of Definition \ref{defi:mpc} to be locally convex, but we have been unable to prove it.

\vspace{0.3cm}
The proof of Theorem \ref{thm:main} consists of two steps. As a first point, we prove the following proposition, which is valid without the $BV$ assumption.

\begin{proposition}\label{prop:sub}
Let $\phi \in \mathbf W^{1,1}(\Omega)$ be a map satisfying \emph{(\textbf{MPC})}. Then $\phi$ is subharmonic in $\Omega$, that is,
\begin{equation}\label{eq:sub}
    - \int_{\Omega} \nabla \phi(x) \cdot \nabla \psi(x) \, \mathrm{d} x \geq 0 \qquad \text{for all $\psi \in C^1_c(\Omega)$}, \qquad \psi \geq 0.
\end{equation}
\end{proposition}
Unfortunately, it is well known that the gradient of subharmonic functions does not need to lie in $BV_{loc}$, as it can be readily seen by looking at the fundamental solution of the Poisson equation
$$
\Delta \Phi = \delta_0, \qquad \Phi = \frac{1}{(d-2) \omega_d}\frac{1}{|x|^{d-2}}.
$$
Therefore presently the $BV_{loc}$ assumption in Theorem \ref{thm:main} cannot be removed.

\vspace{0.3cm}
We will prove Proposition \ref{prop:sub} in Section \ref{sec:sub}, for a more general class of maps which we call \emph{expanding}.

\vspace{0.3cm}
Using Proposition \ref{prop:sub} we can prove Theorem \ref{thm:main}.
\begin{proof}[Proof of Theorem \ref{thm:main}]
By the standard theory of $BV$ functions, see \cite{AFP_book}, since by assumption $\nabla \phi \in BV_{loc}(\Omega; \mathbb R^d)$, we can decompose the derivative $D (\nabla \phi)$ as
$$
D (\nabla \phi) = D^{a} (\nabla \phi) + D^s (\nabla \phi) \in \mathscr M(\Omega, \mathbb R^{d\times d})
$$
where $D^a(\nabla \phi)$ is the absolutely continuous part of $D(\nabla \phi)$ with respect to the Lebesgue measure $\mathscr L^d$, and $D^s(\nabla \phi)$ is the singular part. Here $\mathscr M(\Omega, \mathbb R^{d\times d})$ is the space of locally finite matrix valued measures. In the following we analyze them separately.

\vspace{0.5cm}
\textbf{1.} \emph{(Singular part).} We have that
$$
D^s(\nabla \phi) =  \frac{\mathrm{d} D^s(\nabla \phi) }{\mathrm{d} |D^s(\nabla \phi)|} |D^s(\nabla \phi)|
$$
where $ \frac{\mathrm{d} D^s(\nabla \phi) }{\mathrm{d} |D^s(\nabla \phi)|}$ is the polar matrix, defined for $|D^s(\nabla \phi)|$-a.e. $x$. By the Alberti's rank 1 theorem \cite[Theorem 3.94]{AFP_book} for $|D^s(\nabla \phi)|$ a.e. $x \in \Omega$ there exist $a(x), b(x) \in \mathbb R^n \setminus \{0\}$ such that
$$
 \frac{\mathrm{d} D^s(\nabla \phi) }{\mathrm{d} |D^s(\nabla \phi)|} = a(x) \otimes b(x)
$$
Since the polar matrix $ \frac{\mathrm{d} D^s(\nabla \phi) }{\mathrm{d} |D^s(\nabla \phi)|}$ must be symmetric, we deduce the existence of $\lambda(x) \in \mathbb R \setminus \{0\}$ and $\vec n(x) \in \mathbb R^d \setminus\{0\}$ such that $\lambda(x)\vec n(x)\otimes \vec n(x) = a(x) \otimes b(x)$ for $|D^s(\nabla \phi)|$ a.e. $x \in \Omega$, therefore
\begin{equation}\label{eq:D^s}
D^s(\nabla \phi) =  \lambda \vec n \otimes \vec n \;|D^s(\nabla \phi)|
\end{equation}
and the Laplacian of $\phi$ is 
$$
\Delta \phi = \mathrm{Tr}(D^a(\nabla \phi)) + \lambda |D^s(\nabla \phi)|.
$$
By Proposition \ref{prop:sub} we deduce in particular that 
$$
\lambda(x) \geq 0 \qquad \text{for $|D^s(\nabla \phi)|$-almost every $x \in \Omega$}
$$
therefore the polar $\frac{\mathrm{d} D^s(\nabla \phi) }{\mathrm{d} |D^s(\nabla \phi)|}$ is a positive definite matrix and by \eqref{eq:D^s} $D^s(\nabla \phi)$ is a positive definite matrix valued measure, since for every continuous vector field $\xi \in C_c(\Omega; \mathbb R^d)$, there holds
$$
\int_{\Omega} \xi^\perp(x) \cdot \frac{\mathrm{d} D^s(\nabla \phi) }{\mathrm{d} |D^s(\nabla \phi)|}(x) \cdot \xi(x) \; \mathrm{d}\, |D^s(\nabla \phi)|(x) \geq 0.
$$

\vspace{0.5cm}
\textbf{2.} \emph{(Absolutely continuous part).} We will show that in fact the absolutely continuous part is zero. We will use the following well-known Lusin-Lipschitz properties for $BV$ functions (see \cite[Theorem 5.34]{AFP_book}). In our setting, since $\nabla \phi \in BV(\Omega, \mathbb R^n)$, there exists a constant $\kappa$ (depending only on the dimension $d$) such that for every $\Lambda > 0$ there exists a Lipschitz function $f^\Lambda$ with Lipschitz constant less then $\kappa \Lambda$, such that 
$$
\mathscr L^d \Big( \big\{ x \in \Omega \; | \; \nabla \phi(x) \neq f^\Lambda(x)  \big\} \Big) \leq \frac{\kappa}{\lambda}|D \nabla \phi|(\Omega).
$$
Letting $G^\Lambda := \{ x \in \Omega \; | \; f(x) = \nabla \phi(x)\}$, we also have the property
$$
G^\Lambda_1 \subset G^\Lambda_2 \qquad \text{if $\Lambda_1 \leq \Lambda_2$}.
$$
 We have that $f^\Lambda(x)-\nabla \phi(x) = 0$ for a.e.  $ x \in G^\Lambda$, therefore by the locality of the gradient of $BV$ functions we obtain $$Df^\Lambda \llcorner G^\Lambda = D (\nabla \phi) \llcorner G^\Lambda$$
so that in particular $D^s (\nabla \phi) \llcorner G^\Lambda = 0$, and moreover for the absolutely continuous part there holds
\begin{equation}\label{eq:locgr}
D f^\Lambda(x) = D^a (\nabla \phi)(x) \qquad \text{for $\mathscr L^d$ almost every $x \in G^\Lambda$}.
\end{equation}

Define now the Lipschitz function 
$$
T_t^\Lambda(x) := x + t f^\Lambda(x).
$$
We have now that, using the standard area formula for the Lipschitz function $f^\Lambda$, for every measurable $A \subset \Omega$ there holds
\begin{equation}\label{eq:multeplicity}
\begin{aligned}
    \int_{A \cap G^\Lambda}|\mathrm{det} \, D^a T_t(x)| \mathrm{d} x & =  \int_{A \cap G^\Lambda} |\mathrm{det} \, DT^\Lambda_t(x)| \mathrm{d} x \\
    & = \int_{\mathbb R^d}  \sharp \{ x \in G^\Lambda\cap A \; | \; T_t^\Lambda(x) = y \} \mathrm{d} y\\
    & = \int_{\mathbb R^d}  \sharp \{ x \in G^\Lambda\cap A \; | \; T_t(x) = y \} \mathrm{d} y.
    \end{aligned}
\end{equation}
Notice that for every disjoint measurable $A, B \subset F$, the measure preserving condition \eqref{def:Ttmp} implies that $|T_t(A) \cup T_t(B) | = |T_t(A)| + |T_t(B)|$ and therefore $T_t(A)$ and $T_t(B)$ are essentially disjoint; in particular this yields 
$$
\sharp \{ x \in G^\Lambda\cap A \; | \; T_t(x) = y \} = 1 \qquad \text{for a.e. $y \in \mathbb R^d$}.
$$
In fact, assume by contradiction that for a positive measure set $B \subset \mathbb R^d$ there are at least two elements $x_1(y) \neq x_2(y)$ such that $T_t(x_1(y)) = T_t(x_2(y)) = y$. Then, by the measurable selection theorem, there are two measurable sets $A_1, A_2 \subset \Omega$ such that $T_t(A_1) = B = T_t(A_2)$ up to negligible sets, which is a contradiction.
Therefore we conclude, by letting $\Lambda \to +\infty$, that 
$$
\int_{A}|\mathrm{det} \, D^a T_t(x)| \mathrm{d} x  = \int_{T_t(A)}  1 \mathrm{d} y = |T_t(A)| = |A|
$$
which implies that 
$$
|\mathrm{det} \, D^a T_t(x)| = 1 \qquad \text{for a.e. $x \in \Omega$}.
$$
Therefore we deduce that for almost every $x \in \Omega$
\begin{equation}\label{eq:singlet}
   1 =  | \det (\mathrm{Id}_{\mathbb R^d} + t D^a(\nabla \phi)(x))| = \prod_{i} |(1+ t \lambda_i(x))|
\end{equation}
where $\lambda_i(x)$, $i = 1,\ldots, d$ are the eigenvalues of $D^a(\nabla \phi)(x)$. Repeating the argument for a dense and countable set of times $t_j \geq 0$, and using the continuity in $t$ of the expressions in \eqref{eq:singlet}, we deduce that for $\mathscr L^d$ almost every $x \in \Omega$
\begin{equation}
   1  = \prod_{i} |(1+ t \lambda_i(x))| \qquad \text{for all $t \geq 0$}.
\end{equation}
This clearly implies $\lambda_i(x) = 0$ for all $i = 1, \ldots, d$. Since this holds for $\mathscr L^d$-almost every $x \in \Omega$, we deduce that $D^a(\nabla \phi) = 0$.

\vspace{0.3cm}
Finally, using the first step, we conclude 
$$
D(\nabla  \phi) = D^s(\nabla \phi) \geq 0.
$$
is a positive definite matrix valued measure. It is well known that positive definite distributional Hessian of $\phi$ is equivalent to $\phi$ being locally convex, therefore this concludes the proof.
\end{proof}

\section{Subharmonicity and expanding maps}\label{sec:sub}

The aim of this Section is to provide a proof of Proposition \ref{prop:sub}. We will prove it in the more general setting of \emph{expanding maps}, that we now introduce.

\begin{definition}\label{def:Texp}
We say that a measurable map $T : \Omega \to \mathbb R^d$ is \emph{expanding} if  
\begin{equation}\label{eq:expanding}
   \rho := T_\sharp \mathscr L^d \llcorner \Omega \leq \mathscr L^d
\end{equation}
\end{definition}

\begin{remark}\label{rem:mpexp}
Clearly, if $T$ is measure preserving on $F \subset \Omega$ (Definition \ref{def:Ttmp}) and $F$ is of full measure in $\Omega$, there holds $$T_\sharp \mathscr L^d \llcorner F =  T_\sharp \mathscr L^d \llcorner \Omega = \mathscr L^d \llcorner T(F) \leq \mathscr L^d$$ therefore measure preserving maps are also expanding according to Definition \ref{def:Texp}.
\end{remark}

\vspace{0.3cm}
In the next proposition, we show that if a gradient $\nabla \phi$ is such that $x+ t \nabla \phi(x)$  is expanding for a.e. $t \geq 0$, then $\phi$ is subharmonic:
\begin{proposition}
    Let $\phi: \Omega \to \mathbb R$ be in $\mathbf W^{1,1}(\Omega)$ and assume that $T_t(x) = x + t \nabla \phi(x)$ is expanding for a.e. $t \geq 0$. Then it holds
    \begin{equation}
        \Delta \phi \geq 0 \qquad \text{in $\mathscr D^\prime(\Omega)$}.
    \end{equation}
\end{proposition}
By Remark \ref{rem:mpexp}, this also proves Proposition \ref{prop:sub}.
\begin{proof}
We need to show that for every $0 \leq \psi \in C^1_c(\Omega)$, there holds
\begin{equation}
    \int_\Omega \nabla \phi(x) \nabla \psi(x) d x \leq 0.
\end{equation}

Let $\rho_t = (T_t)_\sharp \mathscr L^d \llcorner \Omega$.
Since $T_t$ is expanding, we have that for all $\psi : \mathbb R^d \to \mathbb R^+$ such that $\psi = 0 $ in $\overline{\Omega^c}$, there holds
\begin{equation}\label{eq:expanding1}
    \int_{\mathbb R^d} \psi(x) d \rho_t \leq  \int_{\mathbb R^d} \psi(x) d \mathscr L^d = \int_{\mathbb R^d} \psi(x) d \rho_0.
\end{equation}
We show that for every $\psi \in C^1_c(\Omega)$ compactly supported in $\Omega$, the function $g^\psi:\mathbb R^+ \to \mathbb R^+$ defined by
$$
g^\psi(t) := \int_{\mathbb R^d} \psi(x) d \rho_t
$$
is Lipschitz and we compute its derivative. We have for all $\delta > 0$
\begin{equation}\label{eq:dct}
\begin{aligned}
    \frac{1}{\delta} |\psi(x+(t+\delta)\nabla \phi(x)) -\psi(x + t\nabla \phi(x)) | \leq L  |\nabla \phi(x)| 
\end{aligned}
\end{equation}
where $L$ is the Lipschitz constant of $\psi$. Then we have, using the definition of $\rho_t$,
$$
\begin{aligned}
    \lim_{\delta \to 0^+}\frac{g^\psi(t+\delta)-g^\psi(t)}{\delta} & = 
    \lim_{\delta \to 0^+} \int_{\mathbb R^d} \psi(x) d \rho_{t+\delta}-\int_{\mathbb R^d} \psi(x) d \rho_t \\
    & = \lim_{\delta \to 0^+} \frac{1}{\delta} \int_\Omega  \psi(x+(t+\delta)\nabla \phi(x)) -\psi(x + t\nabla \phi(x)) d x \\
    & = \int_\Omega  \lim_{\delta \to 0^+}\frac{1}{\delta} \psi(x+(t+\delta)\nabla \phi(x)) -\psi(x + t\nabla \phi(x)) d x\\
    & = \int_\Omega \nabla \psi(x+ t\nabla \phi(x)) \cdot \nabla \phi(x) d x
    \end{aligned}
$$
where the penultimate equality follows from the dominated convergence theorem, which we could apply because of \eqref{eq:dct} since $|\nabla \phi| \in \mathbf L^1(\Omega)$.

Now using \eqref{eq:expanding1} we compute
$$
\begin{aligned}
    0 & \geq g^\psi(t)-g^\psi(0) \\
    & = \int_0^t(g^\psi)^\prime(s) d s \\
    & = \int_0^t \int_\Omega \nabla \psi(x+ s\nabla \phi(x)) \cdot \nabla \phi(x) d x d s.
\end{aligned}
$$
In particular we have
$$
\begin{aligned}
0 & \geq \lim_{t \to 0^+} \frac{1}{t} \int_0^t  \int_\Omega \nabla \psi(x+ s\nabla \phi(x)) \cdot \nabla \phi(x) d x d s \\
& =\lim_{t \to 0^+} \int_\Omega \nabla \phi(x) \cdot \frac{1}{t} \int_0^t  \nabla \psi(x+ s\nabla \phi(x))  d s d x \\
& = \int_\Omega \nabla \phi(x) \nabla \psi(x) d x
\end{aligned}
$$
where the last equality follows from the fact that 
$$
\frac{1}{t}\int_0^t  \nabla \psi(x+ s\nabla \phi(x))  d s \longrightarrow \nabla \psi(x) \qquad \text{strongly in $\mathbf L^1(\Omega)$ as $t \to 0^+$}
$$
because again by \eqref{eq:dct} we have
$$
\begin{aligned}
     \int_\Omega \left| \frac{1}{t} \int_0^t   \nabla \psi(x+ s\nabla \phi(x))- \nabla \psi(x)  d s\right| d x  & \leq  \int_\Omega \frac{1}{t}\int_0^t L s |\nabla \phi(x)| d s d x\\
     & \leq L t \int_\Omega |\nabla \phi(x)| d x.
\end{aligned}
$$
This concludes the proof.

\end{proof}

\section{Optimal transport approach}\label{sec:optproof}
Here we provide a different proof of Proposition \ref{prop:sub} based on optimal transport techniques.

\vspace{0.3cm}
Let us define the measures, for all $t > 0$,
$$
\mu := \mathscr L^d \llcorner \Omega, \qquad \nu_t := (T_t)_\sharp \mathscr L^d \llcorner \Omega.
$$
In this section we need to assume that the first moment is finite, that is
$$
\int |x|^2 \mathrm{d} \mu, \qquad \int |y|^2 \mathrm{d} \nu_t < \infty
$$
which amounts to assume that $T_t \in \mathbf L^2(\Omega, \mathbb R^d)$, or equivalently $\nabla \phi \in \mathbf L^2(\Omega, \mathbb R^d)$, and that 
$$
\int_{\Omega} |x|^2 \mathrm{d} x < \infty.
$$

From \cite{Bre91} we know that there exists a unique convex function $f_t$ such that $(\nabla f_t)_\sharp \mu = \nu_t$.
As above, this means that there exists a set of full measure (which we can take to be equal to $F$, recall Definition \ref{defi:mpc}) such that $\nabla {f_t}$ exists in $F$ and
\begin{equation*}
(\nabla {f_t})_\sharp \mathscr L^d = \mathscr L^d \rest_{F_t}
\end{equation*}
where $F_t = \nabla f_t(F)$.

\vspace{0.3cm}
We recall the Polar factorization theorem by Brenier \cite{Bre91}. 
\begin{theorem}\label{thm:pd}
    Let $\xi : \Omega \to \mathbb R^d$ be a measurable map in $\mathbf L^2(\Omega; \mathbb R^d)$, such that $\xi_\sharp \mathscr L^d \llcorner \Omega$ is absolutely continuous with respect to the Lebesgue measure. There exists a unique decomposition 
    $$
    \xi = \nabla u \circ \Phi
    $$
    were $u: \Omega \to \mathbb R$ is a convex function and $\Phi : \Omega \to \Omega$ is measure preserving (i.e. $\Phi_\sharp \mathscr L^d \llcorner \Omega = \mathscr L^d \llcorner \Omega$). Moreover, $\nabla u$ is the optimal map for the quadratic transport cost  pushing forward $\mathscr L^d \llcorner \Omega$ to $\xi_\sharp \mathscr L^d \Omega$.
\end{theorem}

We have the following proposition.
\begin{proposition}\label{prop:convergencepolar}
    Let $\{X_t\}_{t > 0}$ be a family of maps $X_t \in \mathbf L^2(\Omega, \mathbb R^d)$. Let $\mathbf b \in \mathbf L^2(\Omega, \mathbb R^d)$ be a vector field such that 
    \begin{equation}\label{eq:assb}
    \lim_{t \to 0^+} \int_\Omega \left(\frac{X_t(x)-x}{t}- \mathbf b(x) \right)^2 \mathrm{d} x = 0.
    \end{equation}
    Letting
\begin{equation*}
X_t = \nabla {f_t} \circ \Phi_t
\end{equation*}
be the Polar Decomposition of $X_t$ given by Theorem \ref{thm:pd}, it holds
\begin{equation*}
\nabla h_t := \frac{\nabla {f_t} - \Id}{t} \to \mathtt P_\nabla \b, \quad \frac{\Phi_t - \Id}{t} \to (\Id - \mathtt P_\nabla) \b \quad \text{in $\mathbf L^2$ as $t \to 0^+$}
\end{equation*}
where $\mathtt P_\nabla : \mathbf L^2(\Omega, \mathbb R^d) \to \mathbf L^2(\Omega, \mathbb R^d)$ is the projection on the space of gradients.
\end{proposition}
\begin{proof}
\textbf{1.} Let $h_t$ be defined, up to constants, by 
$$\nabla h_t := \frac{\nabla {f_t} - \Id}{t}.$$
 As a first step, we compute the weak limit of $h_t$ as $t \to 0^+$. 
For any function $\psi \in C^2(\Omega)$, we compute
\begin{equation}\label{eq:ftexp}
\begin{aligned}
 \Bigg|\int_\Omega \psi(x)   & \mathrm{d} [(\nabla f_t)_\sharp \mathscr L^d\rest_\Omega](x) - \int_\Omega \Big( \psi(x) + t \nabla \psi(x) \cdot \nabla h_t(x)\Big)  \mathrm{d} x \Bigg|\\
 & =  \left|\int_\Omega \psi(x + t \nabla h_t(x)) \mathrm{d} x - \int_\Omega \Big( \psi(x) + t \nabla \psi(x) \cdot \nabla h_t(x)\Big)  \mathrm{d} x\right| \\
 & \qquad \leq  \frac{t^2}{2} \|\nabla^2\psi\|_\infty \|\nabla h_t\|^2_{\mathbf L^2}
\end{aligned}
\end{equation}
and
\begin{equation}\label{eq:Xtexp}
\begin{aligned}
 \Bigg| \int_\Omega \psi(x)  & \mathrm{d} [(X_t)_\sharp \mathscr L^d \rest_\Omega](x)-\int_\Omega \Big(\psi(x)  +  \nabla \psi(x) \cdot (X_t(x) - x)\Big)  \mathrm{d} x\Bigg| & \\
 & \Bigg|\int_\Omega \psi(X_t(x))  \mathrm{d} x-\int_\Omega \Big(\psi(x)  +  \nabla \psi(x) \cdot (X_t(x) - x)\Big)  \mathrm{d} x\Bigg|   \\
&\leq \qquad  \frac{t^2}{2} \|\nabla^2\psi\|_\infty \|X_t - \Id \|^2_{\mathbf L^2}.
\end{aligned}
\end{equation}
Since $\nabla f_t$ and $X_t$ have the same image measure
$$
(\nabla {f_t})_\sharp \mathscr L^d \rest_{\Omega} = {X_t}_\sharp \mathscr L^d \rest_{\Omega},
$$
and using that, by  optimality of $\nabla {f_t}$,
\begin{equation*}
\|\nabla h_t\|_{\mathbf L^2(\Omega)} = \bigg\| \frac{\nabla f_t - \Id}{t} \bigg\|_{\mathbf L^2(\Omega)} \leq \frac{1}{t} \|X_t - \Id\|_{\mathbf L^2(\Omega)}
\end{equation*}
we obtain from \eqref{eq:Xtexp}, \eqref{eq:ftexp} that
\begin{equation*}
\begin{aligned}
\Bigg| \int_\Omega \nabla \psi(x) \cdot \bigg( \nabla h_t(x) - \frac{X_t(x) - x}{t} \bigg) {\mathrm{d} x} \Bigg| & = \Bigg|\int_\Omega \nabla \psi(x) \cdot \bigg( \nabla h_t - \mathtt P_\nabla \bigg( \frac{X_t - \Id}{t} \bigg)(x) \bigg) {\mathrm{d} x}\Bigg| \\
& \leq  \; t \, \|\nabla^2 \psi\|_\infty  \bigg\| \frac{X_t - \Id}{t} \bigg\|_{\mathbf L^2(\Omega)}.
\end{aligned}
\end{equation*}
Finally, letting $t \to 0$ we conclude 
\begin{equation}\label{eq:weakht}
    \nabla h_t \longrightarrow \mathtt P_\nabla \b \quad \text{weakly in $\mathbf L^2(\Omega, \mathbb R^d)$ as $t \to 0^+$}
\end{equation}

\vspace{0.5cm}
\textbf{2.}
The aim of this step is to show that $\nabla h_t$ converges strongly to $\mathtt P_\nabla \b$ by showing that 
\begin{equation}\label{eq:norms}
    \limsup_{t \to 0^+} \|\nabla h_t\|_{\mathbf L^2} \leq \|\mathtt P_\nabla \b\|_{\mathbf L^2}.
\end{equation}
In order to prove \eqref{eq:norms}, we first consider the Helmhotz decomposition of $\b$ as
\begin{equation*}
\b = \nabla g + \b_1
\end{equation*}
where $\nabla g$, for some $g \in \mathbf W^{1,2}(\Omega)$, is the $\mathbf L^2$-projection of $\b$ on the subspace of gradients, and $\b_1$ is a divergence free vector field. 

Let $\b_{1,n} \in \mathbf L^2(\Omega, \mathbb R^d)$ be a sequence of smooth divergence free vector fields such that
\begin{equation}\label{eq:bnest}
\|\b_1 - \b_{1,n}\|_{\mathbf L^2(\Omega)} \leq 2^{-n}, \quad \b_{1,n} = 0 \quad \text{on $\partial \Omega$}
\end{equation}
and consider the corresponding flow map $\tilde \Phi^n_t: \Omega \to \Omega$, defined by
\begin{equation*}
\frac{d}{dt} \tilde \Phi^n_t(x) = \b_{1,n}(\tilde \Phi^n_t(x)), \quad \tilde \Phi_n(0,x) = x.
\end{equation*}
Clearly, $\tilde \Phi_t^n$ is measure preserving,
$$
(\tilde \Phi_t^n)_\sharp \mathscr L^2 \llcorner_\Omega = \mathscr L^2 \llcorner_\Omega
$$
since $\b_{1,n}$ are divergence free. Moreover, define  maps $\tilde X^n_t \in \mathbf L^2(\Omega, \mathbb R^d)$ by
\begin{equation*}
\tilde X_n(t,x) = \tilde \Phi^n_t(x) + t \nabla g(\tilde \Phi^n_t(x)).
\end{equation*}
By definition we have
\begin{equation*}
{\tilde X^n_t} = \nabla \tilde {f_t} \circ {\tilde \Phi^n_t}, \quad \tilde f(t,x) := \frac{x^2}{2} + t g(x).
\end{equation*}
Note that there holds by triangular inequality
\begin{equation*}
    \begin{aligned}
        \Bigg( \int_\Omega & \left| {\tilde \Phi^n_t(x)} + t \nabla g({\tilde \Phi^n_t(x)}) - \Big( x + t \b_{1,n}(x) + t \nabla g(x) \Big)\right|^2  \,\mathrm{d} x \Bigg)^{\frac{1}{2}}\\
        & \leq t \Bigg( \int_\Omega \left| \frac{\tilde \Phi^n_t(x) - x}{t} - \b_{1,n}(x)\right|^2  \,\mathrm{d} x \Bigg)^{\frac{1}{2}} + t\Bigg(\int_\Omega | \nabla g({\tilde \Phi^n_t(x)}) -   \nabla g(x)|^2 \mathrm{d} x\Bigg)^{\frac{1}{2}}
    \end{aligned}
\end{equation*}
Since $\b_{1,n}$ is smooth, one clearly has
$$
\lim_{t \to 0^+} \Bigg( \int_\Omega \left| \frac{\tilde \Phi^n_t(x) - x}{t} - \b_{1,n}(x)\right|^2  \,\mathrm{d} x \Bigg)^{\frac{1}{2}} = 0.
$$
We claim that also 
\begin{equation}\label{eq:claim3}
    \lim_{t \to 0^+} \Bigg(\int_\Omega | \nabla g({\tilde \Phi^n_t(x)}) -   \nabla g(x)|^2 \mathrm{d} x\Bigg)^{\frac{1}{2}} = 0.
\end{equation}
This can be easily seen since $\tilde \Phi_t^n$ is measure preserving, via a standard argument as follows. Let $\varepsilon > 0$, then by density there is a Lipschitz function $F: \Omega \to \mathbb R^d$ such that 
$$
\|F- \nabla g\|_{\mathbf L^2} \leq \varepsilon.
$$
Therefore we estimate 
$$
\begin{aligned}
   \lim_{t \to 0^+} \Bigg(\int_\Omega | \nabla g({\tilde \Phi^n_t(x)}) -   \nabla g(x)|^2 \mathrm{d} x \Bigg)^{\frac{1}{2}}& \leq 2\Bigg(\int_\Omega | \nabla g(x) - F(x)|^2 \mathrm{d} x\Bigg)^{\frac{1}{2}} \\
   & +  \lim_{t \to 0^+} \Bigg(\int_\Omega | F({\tilde \Phi^n_t(x)}) -  F(x)|^2 \mathrm{d} x\Bigg)^{\frac{1}{2}} \\
    & \leq 2 \varepsilon + \mathrm{Lip} \, (F)  \lim_{t \to 0^+}  \Bigg(\int_\Omega |\tilde \Phi_t^n(x) - x|^2 \mathrm{d} x \Bigg)^{\frac{1}{2}}= 2\varepsilon
\end{aligned}
$$
Since this holds for every $\varepsilon  >0$, \eqref{eq:claim3} is proved.
Therefore we have
\begin{equation}\label{eq:tildephi}
    \lim_{t \to 0^+} \frac{1}{t}\Bigg(\int_\Omega \left| {\tilde \Phi^n_t(x)} + t \nabla g({\tilde \Phi^n_t(x)}) - \Big( x + t \b_{1,n}(x) + t \nabla g(x) \Big)\right|^2  \,\mathrm{d} x\Bigg)^{\frac{1}{2}} = 0
\end{equation}
We thus compute
\begin{equation}
    \begin{aligned}
        \frac{1}{t}\|X_t - \tilde X_t^n\|_{\mathbf L^2} & \leq \frac{1}{t}\|X_t - (\Id + t \nabla \b)\|_{\mathbf L^2} \\
        & + \frac{1}{t}\|(\Id + t \nabla \b) -(\Id + t (\nabla g + \b_{1,n}))\|_{\mathbf L^2} \\
        & + \frac{1}{t}\| (\Id + t (\nabla g + \b_{1,n})) - \tilde X_t^n\|_{\mathbf L^2}\\
        & \leq o(1) +2^{-n} + o(1)
    \end{aligned}
\end{equation}
where $o(1)$ is a quantity approaching zero with $t \to 0$, and we used respectively \eqref{eq:assb} to estimate the first term, \eqref{eq:bnest} to estimate the second term, and \eqref{eq:tildephi} to estimate the third term.
Then we can estimate the 2-Wasserstein distance between
${X_t}_\sharp \mathscr L^d \rest_{\Omega}$ and ${(\tilde X^n_t)}_\sharp \mathscr L^d \rest_{\Omega}$ as
\begin{equation}\label{eq:cost1}
\begin{aligned}
\dist_{W_2} \big( {X_t}_\sharp \mathscr L^d \rest_{\Omega},{(\tilde X^n_t)}_\sharp \mathscr L^d \rest_{\Omega} \big) & = \|X_t- \tilde X_t^n\|^2_{\mathbf L^2} \leq  t^2 2^{-n+1}
\end{aligned}
\end{equation}
for all $t>0$ and sufficiently small. By the triangular inequality for the 2-Wasserstain distance, we obtain that 
\begin{equation}
    \begin{aligned}
        \dist_{W_2} \big( {X_t}_\sharp \mathscr L^d \rest_{\Omega}, \; \mathscr L^d \rest_\Omega \big)&  \leq \dist_{W_2} \big( {X_t}_\sharp \mathscr L^d \rest_{\Omega},{(\tilde X^n_t)}_\sharp \mathscr L^d \rest_{\Omega} \big)  \\
        & +\dist_{W_2} \big(  \mathscr L^d \rest_\Omega, \; (\tilde X^n_t)_\sharp \mathscr L^d \rest_{\Omega}\big)  \\
         & \leq t^2 2^{-n+1} + t^2\|\nabla g\|_{\mathbf L^2}^2
    \end{aligned}
\end{equation}
where we have used that 
\begin{equation}\label{eq:claim2}
\dist_{W_2} \big(  \mathscr L^d \rest_\Omega, \; (\tilde X^n_t)_\sharp \mathscr L^d \rest_{\Omega}\big) \leq t^2 \|\nabla g\|_{\mathbf L^2}^2
\end{equation}
To see that \eqref{eq:claim2} holds, first notice that since $(\tilde \Phi^n_t)$ is measure preserving, we deduce
\begin{equation}
(\nabla \tilde f_t)_\sharp \mathscr L^d \rest_\Omega = (\nabla \tilde f_t)_\sharp \circ (\tilde \Phi^n_t)_\sharp \mathscr L^d \rest_\Omega =(\tilde X^n_t)_\sharp \mathscr L^d \rest_{\Omega}.
\end{equation}
In particular $\nabla \tilde f_t$ is an admissible map between $\mathscr L^d \rest_\Omega$ and $(\tilde X^n_t)_\sharp \mathscr L^d \rest_{\Omega}$, with cost 
$$
\|\nabla \tilde f_t - \Id\|_{\mathbf L^2}^2 = t^2\|\nabla g\|_{\mathbf L^2}^2
$$
and this proves \eqref{eq:claim2}.
Since $\nabla f_t$ is the optimal map between $\mathscr L^d \rest_\Omega$ and $(X_t)_\sharp \mathscr L^d \rest_\Omega$, we deduce that for all $t> 0$ small
\begin{equation}
    \|\nabla f_t -\Id\|_{\mathbf L^2}^2 = \dist_{W_2} \big( {X_t}_\sharp \mathscr L^d \rest_{\Omega}, \; \mathscr L^d \rest_\Omega \big) \leq t^2 \|\nabla g\|_{\mathbf L^2}^2 + t^2 2^{-n+1}
\end{equation}
We thus conclude that
\begin{equation*}
\limsup_{t \to 0} \int |\nabla h_t|_2^2 \leq \int_\Omega |\nabla g|^2 \mathrm{d} x + 2^{-n}.
\end{equation*}
This holds for all $n$, so that
\begin{equation}\label{eq:norms}
\limsup_{t \to 0} \|\nabla h_t\|_{\mathbf L^2(\Omega)} \leq \|\nabla g\|_{\mathbf L^2(\Omega)}.
\end{equation}
Using that $\nabla h_t \rightharpoonup \nabla g$ weakly, we thus conclude from \eqref{eq:norms} that the convergence is strong by the uniform convexity of $\mathbf L^2$.

\vspace{0.5cm}
\textbf{3.} Notice that we have
\begin{equation}\label{eq:claim1}
\bigg\| \frac{\Phi_t - \Id}{t} \bigg\|_{\mathbf L^2(\Omega)} \leq \|h_t\|_{\mathbf L^2(\Omega)} + \bigg\| \frac{X_t - \Id}{t} \bigg\|_{\mathbf L^2(\Omega)} \leq 2 \bigg\| \frac{X_t - \Id}{t} \bigg\|_{\mathbf L^2(\Omega)}.
\end{equation}
Then we deduce
\begin{equation*}
\begin{split}
\frac{\nabla {X_t} - \Id}{t} &= \frac{\nabla {f_t}\circ \Phi_t - \Phi_t}{t} + \frac{\Phi_t - \Id}{t} = h_t(\Phi_t) + \frac{\Phi_t - \Id}{t},
\end{split}
\end{equation*}
and we conclude that
\begin{equation*}
\lim_{t \to 0} \frac{\Phi_t - \Id}{t} = \lim_{t \to 0} \frac{{X_t} - \Id}{t} - \lim_{t \to 0} h_t \circ \Phi_t = \b - \mathtt P_\nabla \b. \qedhere
\end{equation*}
where we used that since \eqref{eq:claim1} holds, then $h_t \circ \Phi_t \to \mathtt P_\nabla \b$ strongly in $\mathbf L^2$ (this can done as in the proof of \eqref{eq:claim3}).
\end{proof}

Applying the result to our setting, we obtain:

\begin{corollary}
\label{Cor:conser_transprto}
Let $\phi \in \mathbf W^{1,2}(\Omega)$ such that \emph{(\textbf{MPC})} holds. Let $\nabla f_t, \Phi_t$ be the maps in the polar decomposition of $T_t$, i.e.
$$
T_t = \nabla f_t \circ \Phi_t.
$$
Then 
\begin{equation*}
\lim_{t \to 0} \bigg\| \frac{{\nabla f_t} - \Id}{t} - \nabla \phi \bigg\|_{\mathbf L^2} = 0, \quad \lim_{t \to 0} \bigg\| \frac{\Phi_t - \Id}{t} \bigg\|_{\mathbf L^2} = 0,
\end{equation*}
\end{corollary}

We also obtain a different proof of the following Proposition.

\begin{proposition}
\label{Prop:laplacian_1}
Let $\phi \in \mathbf W^{1,2}(\Omega)$ such that \emph{(\textbf{MPC})} holds. Then
$$
\Delta \phi \geq 0.
$$
\end{proposition}

\begin{proof}
Let $\nabla f_t, \Phi_t$ be as above.
Being $\nabla {f_t}$ measure preserving and monotone, hence BV, it holds
\begin{equation*}
\det \big( \Id + t D^a \nabla h_t(x) \big) = 1,
\end{equation*}
for $\mathscr L^d$-a.e. $x \in \Omega$, where we recall 
$$
\nabla h_t = \frac{\nabla f_t - \Id}{t}.
$$
If $\lambda_{i}(t,x)$ are the eigenvalues of $D^a \nabla h_t(x)$ one gets, using the arithmetic-geometric inequality
\begin{equation}
\label{Equa:eigenvalues_nabla_h_t_1}
1 = \bigg( \prod_i (1 + t \lambda_{i}(t,x)) \bigg)^{\frac{1}{d}} \leq \frac{1}{d} \sum_i (1 + t \lambda_{i}(t,x)) = 1 + \frac{t}{d} \sum_i \lambda_{i}(t,x).
\end{equation}
By \eqref{Equa:eigenvalues_nabla_h_t_1}, there holds
$$
\mathrm{Tr}(D^a \nabla h_t) \geq 0
$$
while for the singular part of the Laplacian, one has
\begin{equation*}
t (\Delta h_t)^s = (\Delta {f_t})^s \geq 0,
\end{equation*}
being $f_t$ convex. Hence $\Delta h_t \geq 0$ in the sense of distributions, 
$$
\int_{\Omega} \nabla \psi(x) \cdot \nabla h_t(x) \mathrm{d} x \leq 0 \qquad \text{for all $\psi \in C^1_c(\Omega)$, $\psi \geq 0$}.
$$
Since by Corollary \ref{Cor:conser_transprto} we have the convergence $\nabla h_t \to \nabla \phi$  in $\mathbf L^2$, we deduce for every $\psi \in C^1_c(\Omega)$ that 
$$
0 \geq \lim_{t \to 0^+} \int_{\Omega} \nabla \psi(x) \cdot \nabla h_t(x) \mathrm{d} x = \lim_{t \to 0^+} \int_{\Omega} \nabla \psi(x) \cdot \nabla \phi(x) \mathrm{d} x 
$$
and this proves the result.
\end{proof}

\end{document}